\pdfoutput=1

\documentclass[reqno,a4paper,final]{amsart}

\usepackage[utf8]{inputenc}
\usepackage[T1]{fontenc}
\usepackage[english]{babel}
\usepackage{ifthen}
\usepackage{enumitem}
\usepackage{dsfont}
\usepackage{mathtools}
\usepackage[babel]{csquotes}
\usepackage[backend=biber,style=alphabetic]{biblatex}
\usepackage[protrusion=true,expansion=true,babel=true,final]{microtype}
\usepackage[unicode,bookmarksopen]{hyperref}

\numberwithin{equation}{section} 

\newtheorem{lemma}{Lemma}[section]

\newtheorem{theorem}[lemma]{Theorem}

\theoremstyle{definition}

\newtheorem{remark}[lemma]{Remark}

\newlist{thm_enum}{enumerate}{1}
\setlist[thm_enum]{label=\normalfont(\alph*)}
\newlist{def_enum}{enumerate}{1}
\setlist[def_enum]{label=\normalfont(\roman*)}
\newlist{equiv_enum}{enumerate}{1}
\setlist[equiv_enum]{label=\normalfont(\roman*)}

\newcommand{\IN}{\mathbb{N}}
\newcommand{\IR}{\mathbb{R}}

\newcommand{\IE}{\mathbb{E}}

\newcommand{\abs}[1]{\left\lvert#1\right\rvert}
\newcommand{\normalabs}[1]{\lvert#1\rvert}

\newcommand{\norm}[1]{\left\lVert#1\right\rVert}
\newcommand{\normalnorm}[1]{\lVert#1\rVert}

\newcommand{\R}[2][\empty]{
	\ifthenelse{\equal{#1}{\empty}}
		{\mathcal{R}\left\{#2\right\}}
		{\mathcal{R}_{#1}\left\{#2\right\}}
}

\makeatletter
\newcommand{\LeftEqNo}{\let\veqno\@@leqno}
\makeatother

\renewcommand{\d}{\mathop{}\!d}

\renewcommand{\epsilon}{\varepsilon}
\renewcommand{\phi}{\varphi}

\DeclareMathOperator{\E}{\IE}

\DeclareMathOperator{\Id}{Id}

\DeclareMathOperator{\Ker}{Ker}
\DeclareMathOperator{\Rg}{Rg}

\DeclareMathOperator{\sinc}{sinc}

\allowdisplaybreaks[4]

\DeclareSourcemap{
  \maps[datatype=bibtex, overwrite]{
    \map{
      \step[fieldset=abstract, null]
      \step[fieldsource=entrykey,match=Mer99,final] %
      \step[fieldset=shorthand,fieldvalue=LeM99] %
    }
  }
}

\renewbibmacro{publisher+location+date}{%
	\printlist{publisher}
	\setunit*{\addcomma\space}
	\printlist{location}
  	\setunit*{\addcomma\space}
  	\usebibmacro{date}
\newunit} 

\newbibmacro{string+doi+url}[1]{%
	\iffieldundef{doi}{%
			\iffieldundef{url}{#1}{\href{\thefield{url}}{#1}}%
		}%
		{\href{http://dx.doi.org/\thefield{doi}}{#1}}
}%

\renewbibmacro{in:}{}

\DeclareFieldFormat*{title}{\usebibmacro{string+doi+url}{\mkbibemph{#1}}}
\DeclareFieldFormat*{booktitle}{#1}
\DeclareFieldFormat[article]{volume}{\textbf{#1}\addcomma}
\DeclareFieldFormat[article]{number}{\addspace no.~#1}
\DeclareFieldFormat[article]{pages}{#1}
\DeclareFieldFormat[book,incollection]{volume}{}
\DeclareFieldFormat[book,incollection]{series}{#1\addcomma\space vol.~\thefield{volume}}
\DeclareFieldFormat{journaltitle}{#1} 
\DeclareFieldFormat{url}{}
\DeclareFieldFormat{eprint}{arXiv: \href{http://arxiv.org/abs/#1}{#1}} 

\renewcommand*{\bibfont}{\footnotesize}

\begin{document}

\title[A Short Counterexample to the Inverse Generator Problem]{A Short Counterexample to the Inverse Generator Problem on non-Hilbertian Reflexive $L^p$-spaces}

\begin{abstract}
	We show that the bilateral shift group on $L^p(\IR)$ for $p \in (1, \infty) \setminus \{2\}$ provides a counterexample to the inverse generator problem.
\end{abstract}

\author{Stephan Fackler}
\address{Institute of Applied Analysis, University of Ulm, Helmholtzstr.\ 18, 89069 Ulm}
\email{stephan.fackler@uni-ulm.de}
\thanks{The author is grateful to Jan Rozendaal for introducing him to the inverse generator problem.}
\keywords{inverse generator, transference principle}
\subjclass[2010]{Primary 47D05.}

\maketitle

\section{Introduction}

	Let $A$ be the generator of a bounded strongly continuous $C_0$-semigroup on some Banach space. Further suppose that $A$ is one-to-one and has dense range. Then the inverse $A^{-1}$ is again a well-defined closed densely defined operator with dense range. The \emph{inverse generator problem} asks whether $A^{-1}$ is a generator (of a $C_0$-semigroup) as well.
	
	This problem seems to go back to a question posed by deLaubenfels at the end of \cite{Lau88}. There it was explicitly asked in the Hilbert space case. In this form the question is open up to today. However, it was immediately noticed that the same question is also very interesting in the Banach space setting and that the problem has a positive answer if $A$ is additionally assumed to generate an analytic semigroup bounded on some sector. This follows from the fact that if $A$ is sectorial, then $A^{-1}$ is sectorial of the same angle. However, for generators of non-analytic semigroups no such easy approach exists.
	
	And in fact, a negative answer to the problem on $c_0$ can be found in \cite[pages 343--344]{Kom66}. Probably first being unaware of Komatsu's counterexample, Zwart initiated a systematic study of the problem in \cite{Zwa07}, hereby giving a counterexample on an infinite Hilbert sum of $C_0([0,1))$-spaces. However, the reflexive case needs more effort. Subsequently, a counterexample to the problem on $L^p$-spaces for $p \in (1, \infty) \setminus \{2\}$ was given in \cite{GomZvaTom07}. There it was shown with rather lengthy calculations that Komatsu's generator also is a counterexample on $\ell^p$ for $p \in (1, \infty) \setminus \{2\}$. 
	
	In this short note we give a different and short counterexample to the inverse generator problem on $L^p$ for $p \in (1,\infty) \setminus \{2\}$ based on Fourier analytic methods. The choice of the counterexample is far from arbitrary and is in fact naturally suggested by general principles.

	The following sections~\ref{sec:motivation} on transference and~\ref{sec:the_counterexample} with the actual counterexample can be read independently. However, as just promised, Section~\ref{sec:motivation} gives a clear motivation for the counterexample considered in Section~\ref{sec:the_counterexample}.

\section{The Inverse Generator Problem and the Transference Principle}\label{sec:motivation}

	In this section we show that the choice of the counterexample in Section~\ref{sec:the_counterexample} is natural. Let us for a moment work purely formally. One has $e^{tA^{-1}} = f_t(-A)$ for the function $f_t(z) = e^{-t/z}$ bounded on the right half plane. Note that on the vertical axis one has $f_t(is) = e^{i t/s}$, a strongly oscillating function whose properties will be crucial later. Taking the inverse Laplace transform of $f_t$ in the sense of distributions we obtain
		\begin{align*}
			(\mathcal{L}^{-1} f_t)(s) = \delta_0(s) - \sqrt{t} \frac{J_1(2 \sqrt{ts})}{\sqrt{s}} \eqqcolon \delta_0(s) + b_t(s),
		\end{align*}
	where $J_1$ is the Bessel function of the first kind and of the first order. Hence, rewriting the holomorphic functional calculus in terms of the Hille--Phillips calculus via the Laplace transform~\cite[Lemma~2.12]{Mer99}, one obtains the formula
		\begin{equation}
			\label{eq:formula_inverse}
			e^{tA^{-1}} x = \int_0^{\infty} (\mathcal{L}^{-1} f_t)(s) T(s)x \d s = x + \lim_{R \to \infty} \int_0^{R} b_t(s) T(s)x \d s.
		\end{equation}
	Of course, it is not clear whether the limit of the singular integral on the right hand side exists. 
	
	Now mathematically rigorously, one sees that if $(T(t))_{t \ge 0}$ is exponentially stable and thus $A^{-1}$ is a bounded operator and a fortiori a generator, the integral converges absolutely. In this case one can directly verify that formula \eqref{eq:formula_inverse} holds~\cite[Lemma~3.2]{Zwa07}. One therefore has for a bounded semigroup $(T(t))_{t \ge 0}$ the identity
		\begin{equation}
			\label{eq:inverse_semigroup}
			e^{t(A-\epsilon)^{-1}} x = x + \int_0^{\infty} b_t(s) e^{-\epsilon s} T(s)x \d s.
		\end{equation}
	For the limit $\epsilon \downarrow 0$ one has the following result.
		
	\begin{theorem}\label{thm:inverse_generator}
		Let $A$ be an injective operator with dense range on some Banach space $X$ and assume that $A$ generates a bounded semigroup $(T(t))_{t \ge 0}$. Further, choose $(\epsilon_n)_{n \in \IN} \subset \IR_{>0}$ with $\epsilon_n \to 0$. Then the following are equivalent.
		\begin{equiv_enum}
			\item\label{equiv:bounded} $\norm{\int_0^{\infty} b_t(s) e^{-\epsilon_n s} T(s) \d s}$ is uniformly bounded in $n \in \IN$ and in $t$ from compact subsets of $\IR_{\ge 0}$.
			\item\label{equiv:convergence} For all $x \in X$ the sequence $(e^{t(A-\epsilon_n)^{-1}}x)_{n \in \IN}$ converges uniformly on compact subsets of $\IR_{\ge 0}$.
		\end{equiv_enum}
		If one of the above conditions~\ref{equiv:bounded} or~\ref{equiv:convergence} holds true , then $A^{-1}$ generates the semigroup $S(t)x = \lim_{n \to \infty} e^{t(A-\epsilon_n)^{-1}}x$.
	\end{theorem}
	\begin{proof}
		\ref{equiv:bounded} $\Rightarrow$ \ref{equiv:convergence}: In virtue of \eqref{eq:inverse_semigroup} we see that there exist $M \ge 1$ and $\omega \ge 0$ with
			\[
				\normalnorm{e^{t(A-\epsilon_n)^{-1}}} \le M e^{\omega t} \qquad \text{for all } n \in \IN.
			\]
		Moreover, one has for all $Ax = y \in D(A^{-1}) = \Rg(A)$
			\[
				(A-\epsilon_n)^{-1}y = (A-\epsilon_n)^{-1} Ax = x + \epsilon_n (A - \epsilon_n)^{-1}x \xrightarrow[n \to \infty]{} x = A^{-1}y.
			\]
		Further, $\IR_{> 0} \subset \rho(A^{-1})$ and therefore by \cite[Theorem~III.4.9]{EngNag00} property \ref{equiv:convergence} holds and the limit semigroup is generated by $A^{-1}$.
		
		\ref{equiv:convergence} $\Rightarrow$ \ref{equiv:bounded}: For $t_0 > 0$ the set $\{ e^{t(A-\epsilon_n)^{-1}}x: n \in \IN, t \in [0,t_0] \}$ is bounded for all $x \in X$. Property \ref{equiv:bounded} follows from the uniform boundedness principle and \eqref{eq:inverse_semigroup}.   
	\end{proof}
			
	Now, suppose that $(T(t))_{t \ge 0}$ is a $C_0$-semigroup on $L^p(\Omega)$ for some $p \in (1,\infty)$ which has a \emph{dilation}, which from now one means the following: there exist a $C_0$-group $(U(t))_{t \in \IR}$ with $\norm{U(t)} \le M$ for all $t \in \IR$ on some second $L^p$-space $L^p(\hat{\Omega})$, bounded linear operators $J\colon L^p(\Omega) \to L^p(\hat{\Omega})$ and $Q\colon L^p(\hat{\Omega}) \to L^p(\Omega)$ with $T(t) = QU(t)J$ for all $t \ge 0$. It then follows that for all $t > 0$ and all $\epsilon > 0$
		\[
			\norm{\int_0^{\infty} b_t(s) e^{-\epsilon s} T(s) \d s} \le \norm{J} \norm{Q} \norm{\int_0^{\infty} b_t(s) e^{-\epsilon s} U(s) \d s}.
		\]
	Further, it follows from the transference principle of Coifman--Weiss \cite[Theorem~2.4]{CoiWei76} that
		\[
			\norm{\int_0^{\infty} b_t(s) e^{-\epsilon s} U(s) \d s} \le M^2 \norm{\int_0^{\infty} b_t(s) e^{-\epsilon s} S(s) \d s},
		\]
	where $(S(s))_{s \in \IR}$ is the shift group on $L^p(\IR)$. Hence, the estimate in Theorem~\ref{thm:inverse_generator} for the shift group would be automatically inherited by all semigroups with a dilation. We now consider different instances of this principle.
	
	We start with the Hilbert space case. Of course, every Hilbert space is isomorphic to some $L^2$-space and so the above considerations apply. For Hilbert spaces it is known that a $C_0$-semigroup has a dilation if and only if it is similar to a contractive semigroup (for the discrete version see~\cite[Remark~4.3]{ArhMer14}, for a proof of the non-standard part for the continuous version follow the arguments in the proof of~\cite[Theorem~5.1]{Fac14c}). Further using that $-d/dx$ is unitarily equivalent to the multiplication operator with $x \mapsto -ix$, one can verify the required estimates of Theorem~\ref{thm:inverse_generator} if $A = -d/dx$ on $L^2(\IR)$. Hence, we reproduce \cite[Lemma~4.1]{Zwa07}.
	
	\begin{theorem}
		Let $A$ be an injective operator on some Hilbert space that generates a $C_0$-semigroup that is similar to a contractive semigroup. Then the inverse $A^{-1}$ is a generator.
	\end{theorem}
	
	\begin{remark}
		The above theorem immediately follows from the Lumer--Phillips theorem as already noted by Zwart. However, our approach has the advantage that it gives a common framework for the reflexive $L^p$-scale and in this way immediately guides us to a counterexample on $L^p$ for $p \neq 2$.
	\end{remark}
	
	Further, for $p \in (1,\infty)$ every positive contractive semigroup on $L^p$ has a dilation by \cite{Fen97}. This strongly suggests that one should study the inverse of the generator of the shift group on $L^p(\IR)$, the task of the next section.

\section{The Inverse Generator of the Shift Group}\label{sec:the_counterexample}

	We now consider the shift group $(S(t)f)(s) = f(s-t)$ on $L^p(\IR)$ for $p \in (1, \infty)$. Its generator is $A_p = -\frac{d}{dx}$ with domain $D(A_p) = W^{1,p}(\IR)$. It is easy to see that $A_p$ is densely defined, injective and has dense range (in fact, the last two are equivalent because of the direct decomposition $L^p(\IR) = \Ker A_p \oplus \overline{\Rg}(A_p)$~\cite[Proposition~15.2]{KunWei04}). Hence, $A_p^{-1}$ is a well-defined closed operator. In this section we answer the question whether $A_p^{-1}$ generates a $C_0$-semigroup on $L^p(\IR)$.
	
	Note that under the Fourier transform $\mathcal{F}$ one has $A_pf = \mathcal{F}^{-1} (x \mapsto -ix (\mathcal{F}f)(x))$ for all $f \in \mathcal{S}(\IR)$, the space of Schwartz functions. In other words, the operator $A_p$ acts as multiplication with the function $x \mapsto -ix$ on the Fourier side. This implies that if $A_p^{-1}$ generates a $C_0$-semigroup on $L^p(\IR)$, then for all $t > 0$ 
		\[
			\mathcal{S}(\IR) \ni f \mapsto \mathcal{F}^{-1}(x \mapsto e^{i (1/x) t} (\mathcal{F}f)(x)) \in \mathcal{S'}(\IR)
		\]
	induces a bounded operator on $L^p(\IR)$, where $\mathcal{S}'(\IR)$ is the space of tempered distributions.
	
	We now put this problem in the appropriate Fourier analytic context. For details and proofs we refer to \cite{Gra08}, in particular to Section~2.5 on multipliers. We denote by $\mathcal{M}_p$ the space of all bounded measurable functions $m\colon \IR \to \IR$ for which the linear operator
		\[
			T_m\colon f \mapsto \mathcal{F}^{-1}(m \cdot \mathcal{F}f) \in \mathcal{S}'(\IR)
		\]
	initially only defined on the space of Schwartz functions induces a bounded operator $L^p(\IR) \to L^p(\IR)$. By setting $\norm{m}_{\mathcal{M}_p} \coloneqq \norm{T_m}_{\mathcal{B}(L^p(\IR))}$ the bounded Fourier multipliers $\mathcal{M}_p$ on $L^p(\IR)$ form a Banach algebra which has the following properties for a given $m \in L^{\infty}(\IR)$:
	
	\begin{enumerate}
		\item\label{multiplier:conjugate} Let $\frac{1}{p} + \frac{1}{p'} = 1$ be conjugated Hölder indices. Then $m \in \mathcal{M}_p$ if and only if $m \in \mathcal{M}_{p'}$. If one of these conditions holds, then one has $\norm{m}_{\mathcal{M}_p} = \norm{m}_{\mathcal{M}_{p'}}$.
		\item Let $\tilde{m}(x) = m(-x)$ be the reflection of $m$. Then $m \in \mathcal{M}_p$ if and only if $\tilde{m} \in \mathcal{M}_p$. If one of these conditions holds, then one has $\norm{m}_{\mathcal{M}_p} = \norm{\tilde{m}}_{\mathcal{M}_p}$.
		\item The set of all $p \in [1, \infty)$ for which $m \in \mathcal{M}_p$ is an interval containing 2.
	\end{enumerate}

	The last point follows from the Riesz--Thorin interpolation theorem. Further, $\mathcal{M}_2 = L^{\infty}(\IR)$ by Plancherel's theorem. This implies that $A_2^{-1}$ generates a contractive $C_0$-semigroup on $L^2(\IR)$. 
	
	Let us now come to the case $p \neq 2$ for which we will see that $A_p^{-1}$ does not generate a $C_0$-semigroup on $L^p(\IR)$. The proof uses the following estimate for oscillatory integrals, the van der Corput lemma \cite[page~332]{Ste93}.
	
	\begin{lemma}\label{lem:van_der_corput}
		Let $\Phi\colon (a,b) \to \IR$ be a smooth phase function on a non-empty open interval with $\normalabs{\Phi^{(k)}(x)} \ge \rho$ for some natural number $k \ge 2$ and $\rho > 0$. Then there exists a universal constant $M_k > 0$ such that
			\[
				\abs{\int_{a}^b e^{i \Phi(x)} \d x} \le M_k \rho^{-1/k}.
			\]
	\end{lemma}
	
	We are now ready to give the following negative answer to the inverse generator problem on reflexive non-Hilbertian $L^p$-spaces.
	
	\begin{theorem}\label{thm:inverse_generator_shift}
		The operator $A_p^{-1}$ generates a $C_0$-semigroup on $L^p(\IR)$ if and only if $p = 2$.
	\end{theorem}
	\begin{proof}
		By the above considerations it suffices to show that for $p \in (2, \infty)$ the multiplier $m(x) = e^{i/x}$ is not bounded on $L^p(\IR)$. Assume that this would be the case. Then its inverse $m^{-1}(x) = e^{-i/x} = \tilde{m}(x)$ is a bounded multiplier as well. Hence, there exists a constant $C \ge 0$ such that for all $f \in L^p(\IR)$ one has
			\begin{equation}
				\label{eq:comparable}
				C^{-1} \norm{f}_p \le \norm{T_m f}_p \le C \norm{f}_p.
			\end{equation}
		Now let $I = [a,b]$ for $a < b$ be a non-empty interval and set $f_I = \mathcal{F}^{-1} \mathds{1}_{I}$. A short calculation yields
			\begin{align*}
				f_I(x) & = \int_a^b e^{2\pi i xy} \d y = \frac{1}{2\pi i x} e^{\pi i (a+b)x} \left( e^{\pi i (b-a)x} - e^{\pi i (a-b)x} \right) \\
				& = e^{\pi i (a+b)x} \frac{\sin(\pi (b-a) x)}{\pi x} = e^{\pi i (a+b)x} \abs{I} \sinc(\abs{I} x).
			\end{align*}
		Further, one has
			\begin{align*}
				\int_{-\infty}^{\infty} \abs{f_I(x)}^p \d x = \int_{-\infty}^{\infty} \abs{I}^p \abs{\sinc(\abs{I}x)}^p \d x = \abs{I}^{p-1} \int_{-\infty}^{\infty} \abs{\sinc(x)}^p \d x = \abs{I}^{p-1} N_p^p.
			\end{align*}
		Hence, $\norm{f_I}_p = \abs{I}^{(p-1)/p} N_p$ for a universal $p$-dependent constant $N_p > 0$. We now estimate $\norm{T_m f_I}_p$. For this observe that on the one hand we have by Plancherel's theorem
			\begin{align*}
				\norm{T_m f_I}_2 = \norm{m \mathds{1}_I}_2 = \norm{\mathds{1}_I}_2 = \abs{I}^{1/2}.
			\end{align*}
		On the other hand one may apply the van der Corput lemma (Lemma~\ref{lem:van_der_corput}) to the phase $\Phi_y(x) = \frac{1}{x} + 2\pi xy$, where $y \in \IR$. Then $\Phi_y''(x) = \frac{2}{x^3}$ independently of $y$ and consequently, if $I$ chosen such that $\normalabs{\Phi_y''(x)} \ge \rho$ for all $x \in I$, one obtains for all $y \in \IR$
			\begin{align*}
				\abs{(T_m f_I)(y)} = \abs{\int_{I} e^{i \Phi_y(x)} \d x} \le M \rho^{-1/2}
			\end{align*}			
		for some universal constant $M \ge 0$. This shows $\norm{T_m f_I}_{\infty} \le M \rho^{-1/2}$. Using the interpolation inequality, we now obtain
			\begin{align*}
				\norm{T_m f_I}_p \le \norm{T_m f_I}_{\infty}^{1-\frac{2}{p}} \norm{T_m f_I}_2^{\frac{2}{p}} \le M^{1-\frac{2}{p}}  \rho^{\frac{1}{p} - \frac{1}{2}} \abs{I}^{\frac{1}{p}}.
			\end{align*}
		Inserting this estimate in \eqref{eq:comparable}, we obtain for all intervals $I$ the inequality
			\begin{equation}
				\label{eq:bounded}
				C^{-1} \abs{I}^{\frac{p-1}{p}} N_p \le M^{1-\frac{2}{p}} \rho^{\frac{1}{p} - \frac{1}{2}} \abs{I}^{\frac{1}{p}} \qquad \Leftrightarrow \qquad \abs{I}^{\frac{p-2}{p}} \rho^{\frac{1}{2} - \frac{1}{p}} \le K
			\end{equation}
		for some universal constant $K > 0$. Now, we choose the order of the size of the interval $I$ maximal under the constraint $\normalabs{\Phi''_y(x)} \ge \rho$. To be explicit, one may use $I = [2^{-1} \rho^{-1/3},\rho^{-1/3}]$. Then $\abs{I} = 2^{-1} \rho^{-1/3}$ and \eqref{eq:bounded} implies that the function 
			\[
				\rho \mapsto \rho^{-\frac{1}{3} (1- \frac{2}{p})} \rho^{\frac{1}{2} - \frac{1}{p}} = \rho^{\frac{1}{6} - \frac{1}{3} \frac{1}{p}}
			\]
		is bounded on $(0, \infty)$. However, the exponent is positive because of $p > 2$ and we have therefore obtained a contradiction.
	\end{proof}
	
	\begin{remark}
		Counterexamples involving shifts have already been studied by different authors. In fact, the example given in~\cite{GomZvaTom07} is of the form $-\Id + S$, where $S$ is the right shift on $\ell^p(\IN)$. Further, it was shown in \cite{Lau09} that the inverse of $-d/dx$ on the closure of its image in $L^1([0,\infty))$ is not a generator. In a similar direction it was noticed in the comments of \cite{Gom11} that the inverse generator of the nilpotent shift on $L^p([0,1])$ generates a $C_0$-semigroup which is not bounded.
	\end{remark}

	\printbibliography

\end{document}